\documentclass[12pt]{article}
\usepackage{amssymb,amsthm,amsmath}
\usepackage{graphicx}
\usepackage{epsfig}
\usepackage{epstopdf}
\usepackage{subfigure}
\usepackage{fullpage}
\usepackage{hyperref}

\newtheorem{theorem}{Theorem}

\newtheorem{prop}[theorem]{Proposition}
\newtheorem{claim}[theorem]{Claim}

\newtheorem{obs}[theorem]{Observation}

\newtheorem{problem}[theorem]{Problem}

\newcommand{\NPc}{{\sc NP}-complete }
\newcommand{\NPcx}{{\sc NP}-complete}

\title{Topological orderings of weighted directed acyclic graphs}
\author{
	{D\'aniel Gerbner}\thanks{Hungarian Academy of Sciences, Alfr\'ed R\'enyi Institute of Mathematics, P.O.B. 127, Budapest H-1364, Hungary. Email gerbner.daniel@renyi.mta.hu. Phone number +3614838362. Research supported by Hungarian National Science Fund (OTKA), grant PD 109537.}
	\and
	Bal\'azs Keszegh\thanks{Hungarian Academy of Sciences, Alfr\'ed R\'enyi Institute of Mathematics, P.O.B. 127, Budapest H-1364, Hungary. Research supported by Hungarian National Science Fund (OTKA), grant PD 108406 and grant NN 102029 (EUROGIGA project GraDR 10-EuroGIGA-OP-003) and the J\'anos Bolyai Research Scholarship of the Hungarian Academy of Sciences}
	\and
	Cory Palmer\thanks{Department of Mathematical Sciences,
		University of Montana, Missoula, Montana 59801, USA. Research supported by Hungarian National Science Fund (OTKA), grant NK 78439.}
	\and
	D\"om\"ot\"or P\'alv\"olgyi\thanks{E\"otv\"os Lor\'and University, Faculty of Science, Department of Computer Science, P\'azm\'any P\'eter 1/C, Budapest H-1117, Hungary. Research supported by Hungarian National Science Fund (OTKA), grant PD 104386 and the J\'anos Bolyai Research Scholarship of the Hungarian Academy of Sciences.}
}

\begin{document}
	
	\maketitle

\begin{abstract}
 We call a topological ordering of a weighted directed acyclic graph \emph{non-negative} if the sum of weights on the vertices in any prefix of the ordering is non-negative.
 We investigate two processes for constructing non-negative topological orderings of weighted directed acyclic graphs. The first process is called a \emph{mark sequence} and the second is a generalization
 called a \emph{mark-unmark sequence}. We answer a question of Erickson by showing that every non-negative topological ordering that can be realized by a mark-unmark sequence can also be realized
 by a mark sequence. We also investigate the question of whether a given weighted directed acyclic graph has a non-negative topological ordering. We show that even in the simple case
 when every vertex is a source or a sink the question is \NPcx.
\end{abstract}


\section{Introduction}

A \emph{directed acyclic graph} (or DAG) is a directed graph with no directed cycles.
A subset $M$ of vertices of $G$ is \emph{outdirected} if every edge 
between $M$ and $V(G) \setminus M$ is directed towards $V(G) \setminus M$ (i.e., edges directed towards $M$ are contained in $M$).
A \emph{prefix of length $k$} of a sequence $s$ is the subsequence of the first $k$ terms of $s$.
A \emph{topological ordering} of a DAG $G$ is an ordering of the vertices of $G$ such that every prefix of the ordering is outdirected.
The following two processes yield topological orderings of a given DAG $G$ with $n$ vertices.

A \emph{mark sequence} of $G$ is a sequence $M_1,M_2,\dots, M_n$ of subsets of $V(G)$ formed in the following way:
first choose an arbitrary source $v$ and put $M_1 = \{v\}$, i.e., \emph{mark} $v$ in step $1$. 
For $i=2,3,4,\dots, n$, choose a vertex $u \not \in M_{i-1}$ such that $\{u\} \cup M_{i-1}$ is outdirected and put $M_i= \{u\} \cup M_{i-1}$, i.e., \emph{mark} $u$ in step $i$.

A \emph{mark-unmark sequence} of $G$ is a sequence of subsets of $V(G)$ formed in the following way:
first choose an arbitrary source $v$ and put $M_1 = \{v\}$, i.e., \emph{mark} $v$ in step $1$. For $i=2,3,4,\dots, n$ 
either (i) choose a vertex $u \not \in M_{i-1}$ such that $\{u\} \cup M_{i-1}$ is outdirected and put $M_i= \{u\} \cup M_{i-1}$, i.e., \emph{mark} $u$ in step $i$
or (ii) choose a vertex $u \in M_{i-1}$ such that $M_{i-1} \setminus \{u\}$ is outdirected and put $M_i = M_{i-1} \setminus \{u\}$, i.e., \emph{unmark} $u$ in step $i$.
This process stops when $M_i=V(G)$.

Clearly, mark-unmark sequences are a generalization of mark sequences.
Because we only mark a vertex if the new set $M_i$ is outdirected, we get a topological ordering by arranging the vertices
of $G$ by the last step in which they were marked in the mark-unmark sequence.
In particular, the ordering of elements given by a mark sequence is simply a topological ordering.

A DAG $G$ is called \emph{weighted} if there is an assignment of real numbers to each vertex of $G$.
We call a topological ordering \emph{non-negative} if the sum of the weights of the vertices in every prefix is non-negative.
Similarly a mark-unmark (or mark) sequence is \emph{non-negative} if at each step the sum of the weights in $M_i$ is non-negative.

Clearly, a non-negative mark sequence is equivalent to a non-negative topological ordering.
However, a non-negative mark-unmark sequence may give a negative topological ordering (we use \emph{negative} in place of ``not non-negative.'').
For example, let $G$ be a weighted DAG on four vertices $\{a,b,c,d\}$ with a single edge $bc$ and weights $w(a)=w(c)=w(d)=1$, $w(b)=-1$.
Consider the following non-negative mark-unmark sequence: mark $a$, $b$, $c$, then unmark $a$, then mark $d$ and $a$.
This gives the topological ordering $b,c,d,a$, which is negative. This suggests the following question of
Erickson\footnote{Positive topological ordering, take 2, http://cstheory.stackexchange.com/questions/1399}: is there a weighted DAG $G$ that has a non-negative mark-unmark sequence but no non-negative mark sequence? 

We answer this question in the negative with the following theorem.

\begin{theorem}\label{main-one}  
	If a weighted DAG $G$ has a non-negative mark-unmark sequence, then $G$ also has a non-negative mark sequence.
\end{theorem}

This problem was motivated by a question of Eppstein\footnote{Positive topological ordering, http://cstheory.stackexchange.com/questions/1346}, which asked to determine the complexity to decide whether a weighted DAG $G$ has a non-negative topological ordering.
His motivation was related to abstract Fr\'echet distance problems.
This question turned out to be practically equivalent to an \NPc problem of Garey and Johnson \cite{GJ}, called SEQUENCING TO MINIMIZE MAXIMUM CUMULATIVE COST, which we will not define here in its full generality.
In fact, the problem is \NPc even in the following special case.

\begin{theorem}\label{main-two}
	Let $G$ be a weighted DAG such that every vertex is either a source or a sink.
	Deciding whether $G$ has a non-negative topological ordering is \NPcx.
\end{theorem}

The proof of hardness is through a series of reductions, which have been noticed/proved by different people (including the authors) and are hard to gather from the internet, so we include the full proof of Theorem~\ref{main-two} in Section~\ref{NPC-section}.
Theorem~\ref{main-one} is proved in Section~\ref{MU-section}.


\section{Marking and Unmarking}\label{MU-section}
In this section we prove Theorem \ref{main-one}.\footnote{A sketch of this proof posted by the fourth author can be found online at http://cstheory.stackexchange.com/questions/1399}
In particular, given a weighted DAG $G$ and a non-negative mark-unmark sequence, we will construct a non-negative mark sequence for $G$. We begin with some definitions.
By $w(X)$ we denote the sum of the weights of the elements of a set of vertices $X$. 
We say that a set $Y \subseteq X$ is $X$-\emph{indirected} if every edge between $Y$ and $X \setminus Y$ is directed towards $Y$.
Similarly, say that a set $Y \subseteq X$ is $X$-\emph{outdirected} if every edge between $Y$ and $X \setminus Y$ is directed towards $X \setminus Y$.
For simplicity, we call a set of vertices $Y$ of a DAG $G$ \emph{outdirected} (\emph{indirected}) if $Y$ is $V(G)$-outdirected ($V(G)$-indirected). Note that this definition corresponds to the definition of indirected given in the previous section.
Figure~1 (a) gives an example of $X$-indirected and $X$-outdirected sets.

\begin{figure}[t]\label{figure1}
	\centering
	\subfigure[$Y_{in}$ is an  $X$-\emph{indirected} set, $Y_{out}$ is an $X$-\emph{outdirected} set.]{\label{fig:inoutdirected}
		\includegraphics[scale=1]{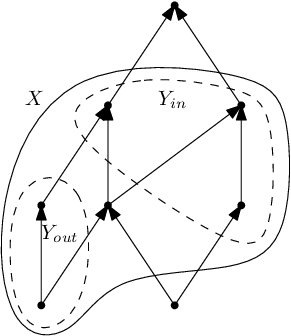}     
	}        
	\hskip 40mm
	\subfigure[12 steps of a mark-unmark sequence, number $j$ (resp. $-j$) denotes that the vertex was added (resp. removed) in step $j$.]{\label{fig:markunmarkexample}
		\includegraphics[scale=1]{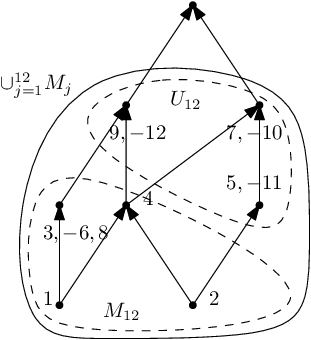}          
	} 
	\caption{Examples}
\end{figure}

\begin{proof}[Proof of Theorem~\ref{main-one}]
	Let $G$ be a weighted DAG with a non-negative mark-unmark sequence. Let $M_1,M_2,$ $\ldots,$  $M_t$ be a mark-unmark sequence with at least one unmark step (otherwise we are done) of minimum length.
	For $i\in [t]$,\footnote{Here (and later) $[t]$ stands for $\{1,\ldots,t\}$.}
	put $U_i=(\cup_{j=1}^{i-1} M_j) \setminus M_i$, i.e., the set of elements that were marked in the first $i-1$ steps, but later became unmarked (in one of the first $i$ steps).
	Note that $U_i$ is $(M_i\cup U_i)$-indirected. Figure~1 (b) gives an example of the steps of a mark-unmark sequence and of $U_i$.
	
	\begin{claim}\label{posi} 
		$w(U_i)>0$ for all $i$.
	\end{claim}
	
	\begin{proof}
		We prove the stronger statement that the weight of any $U_i$-indirected set is positive ($U_i$ is clearly $U_i$-indirected). Suppose the statement is false and let $X$ be a minimal counterexample, i.e., $X$ is $U_i$-indirected and $w(X) \leq 0$. 
		If $Y$ is a non-empty $X$-outdirected set, then $X\setminus Y$ is $U_i$-indirected and a proper subset of $X$, hence $w(X\setminus Y) > 0$ (by the minimality of $X$). 
		
		If $w(Y)\ge 0$, then $w(X)=w(Y)+w(X\setminus Y)>0$ which is a contradiction. Thus we can suppose that for every $X$-outdirected set $Y$ we have $w(Y)<0$. 
		Now we construct a new sequence $M'_1,M'_2,\dots, M'_{t'}$. We start with the original sequence $M_1,M_2,\dots, M_t$. We remove each of $M_1,M_2,\dots, M_{i-1}$ that involves marking or unmarking an element of $X$, and we delete elements of $X$ from the remaining ones. After that we put $M_i,M_{i+1},\dots, M_{t}$ to the end of the sequence. For example, if $X=\{b\}$ and $i=6$ the sequence  $M_1=\{a\}$, $M_2=\{a,b\}$, $M_3=\{a,b,c\}$, $M_4=\{a,b,c,d\}$, $M_5=\{a,c,d\}$, $M_6=\{a,c,d,e\}$, $M_7=\{a,b,c,d,e\}$ becomes  $M_1'=\{a\}$, $M_2'=\{a,b\}$, $M_3'=\{a,b,c\}$, $M_4'=\{a,b,c,d\}$, $M_5'=\{a,b,c,d,e\}$, as we skip the second step (marking $b$) and the fifth step (unmarking $b$).
		
		We claim that this new sequence is also a mark-unmark sequence. First, note that the elements of $X$ are in $U_i$ and are therefore marked at some step after $i$, i.e., each element of $X$ will eventually be marked in the new sequence. Now we show that every $M_j'$ is outdirected. Indeed, if it is not outdirected, then there is an edge $uv$ with $u\not\in M_j'$ and $v\in M_j'$. The set $M_j$ is outdirected and contains $M_j'$ (thus contains $v$), therefore $u\in M_j$. Thus $u\in X$. Now, as $X$ is $U_i$-indirected, either $v\in X$ (which contradicts $v\in M_j'$), or $v\not\in U_i$. But $v\in M_j$ implies it is marked before the $i$th step in the original sequence, hence $v\not\in U_i$ is possible only if it is never unmarked, i.e., $v\in M_i$. However, $u\in X\subset U_i$ implies $u\not\in M_i$, which contradicts the outdirected property of $M_i$.
		
		Now we show that the new mark-unmark sequence $M'_1,M'_2,\dots, M'_{t'}$ is non-negative. Indeed, let $X_j'=X\cap M_j$, then $X_j'$ is an $X$-outdirected set, thus $w(X_j') < 0$. Now $w(M_j')=w(M_j)-w(X_j')> w(M_j)\ge 0$. This new non-negative mark-unmark sequence is shorter than the original sequence, which is of minimum length, a contradiction.
		This completes the proof of Claim~\ref{posi}.
	\end{proof}
	
	We now construct a new sequence by starting with the original mark-unmark sequence $M_1,M_2,\dots, M_t$ and skipping every step where a vertex is unmarked or marked beyond the first marking. Also, when we skip unmarking an element $x$, we add it to every later set.
	Let $M''_1,M''_2,\dots, M''_t$ be the new sequence.
	For example, the sequence  $M_1=\{a\}$, $M_2=\{a,b\}$, $M_3=\{a,b,c\}$, $M_4=\{a,b,c,d\}$, $M_5=\{a,c,d\}$, $M_6=\{a,c,d,e\}$, $M_7=\{a,b,c,d,e\}$ becomes  $M_1''=\{a\}$, $M_2''=\{a,b\}$, $M_3''=\{a,b,c\}$, $M_4''=\{a,b,c,d\}$, $M_5''=\{a,b,c,d,e\}$, as we skip the fifth step (unmarking $b$) and the seventh step (marking $b$ again).
	
	We claim that $M''_1,M''_2,\dots, M''_t$ is a mark sequence. Clearly, every vertex will be marked at some point as every vertex will be marked at the end of the original sequence. Furthermore, 
	every $M_i''$ is outdirected. Indeed, if $M''_i$ is not outdirected, then there is an edge $uv$ with $u\not\in M_i''$ and $v\in M_i''$. 
	But for $v$ to be marked, $u$ had to have been marked in a previous step. This is a contradiction, thus $M''_i$ must be outdirected.
	Finally, to show the sequence is non-negative we must prove that $w(M_i'')$ is non-negative for all $i$. For each $i$, there is some $j\geq i$ such that $M''_i = M_j \cup U_j$.
	The original sequence is non-negative, and $M_j$ and $U_j$ are disjoint by definition, thus, by Claim~\ref{posi}, we have $w(M_i'')=w(M_j)+w(U_j)>w(M_j)\ge 0$ and therefore the mark sequence is non-negative.
	This completes the proof of Theorem~\ref{main-one}.
\end{proof}

Note that with this method we prove a stronger statement. Suppose we want to mark only a certain subset of the vertices, and all the other vertices are used only to help achieve this. Consider, for example, the weighted DAG $G$ mentioned in the introduction with four vertices $\{a,b,c,d\}$, a single edge $bc$ and weights $w(a)=w(c)=w(d)=1$, $w(b)=-1$. Suppose our goal is to mark the subset $\{b,c\}$. Clearly, to mark the subset $\{b,c\}$ we must unmark vertices. For example we can mark $a$, mark $b$, mark $c$ and unmark $a$ to get the desired subset. We show that unmarking is needed only to reduce the set of marked vertices to the desired set, i.e.,  we never perform another mark after the first unmark.

We call a mark-unmark sequence \emph{partial} if it is the same as a mark-unmark sequence except it stops after $t$ steps and $M_t$ may be a proper subset of $V(G)$.

\begin{prop}\label{partial} Let $G$ be a weighted DAG, and suppose $M_1,M_2,\dots,M_t$ is a non-negative partial mark-unmark sequence which stops after $t$ steps. Then there is a non-negative partial mark-unmark sequence $M'_1,M'_2,\dots,M'_{j+k}$ such that the first $j$ steps are markings, the last $k$ steps are unmarkings, and $M_{j+k}'=M_t$.
\end{prop}

We omit some details from the proof of Proposition~\ref{partial} as it is very similar to the proof of Theorem~\ref{main-one}.

\begin{proof}[Proof sketch] 
	For $i\in [t]$, put $U_i=(\cup_{j=1}^{i-1} M_j) \setminus M_i$, i.e., the set of elements that have been unmarked in any of the first $i$ steps.
	As in the proof of Theorem~\ref{main-one}, we have that the weight of any $U_i$-indirected set is positive (from the proof of Claim~\ref{posi}). Let $U$ be the set of vertices that are unmarked at some step in the mark-unmark sequence $M_1,M_2,\dots,M_t$.
	
	We now construct a new sequence by starting with the original mark-unmark sequence $M_1,M_2,\dots, M_t$ and skipping every step where a vertex is unmarked or marked beyond the first marking.
	Let $M'_1,M'_2,\dots, M'_j$ be the new sequence. Then add $|U|=k$ steps to the resulting sequence where we unmark the vertices of $U$ in the order in which their first unmarking occurred in the original sequence. In other words, when an element of $U$ is unmarked first, we move that step to the end of the sequence, and skip all other steps where that element is chosen to be marked or unmarked. The resulting sequence is $M'_1,M'_2,\dots,M'_{j+k}$.
	
	To finish the proof we need to show that $M'_1,M'_2,\dots,M'_{j+k}$ satisfies the definition of a mark-unmark sequence (we omit the details) and that the sequence is non-negative. To see that it is non-negative, observe that after any step of the new sequence the set of marked elements is the same as the set of marked elements after some step in the original sequence, with the addition of an $U_i$-indirected set (which is positive by the proof of Claim~\ref{posi}).
	This completes the proof of Proposition~\ref{partial}.
\end{proof}


\section{\NPcx ness}\label{NPC-section}

In this section we will prove Theorem~\ref{main-two}. The proof is by a series of reductions of the original problem to a known \NPc problem. We restate the decision problem in Theorem~\ref{main-two} here.

\begin{problem}\label{main-two-prob}
	Given a weighted DAG $G$ such that every vertex is either a source or a sink, determine whether $G$ has a non-negative topological ordering.
\end{problem}

First we show that there is a reduction from the following problem posed by Rote\footnote{Graphs extendable to a uniquely matchable bipartite graph, Egres Open collection,	http://lemon.cs.elte.hu/egres/open} to Problem~\ref{main-two-prob}.

\begin{problem}\label{graf}
	Given a balanced bipartite graph $G$, determine whether edges can be added to $G$ to create a bipartite graph with a unique perfect matching.
\end{problem}

We will need the following well-known observation. We include the easy proof for the sake of completeness.

\begin{obs}\label{parositas} 
	Let $G$ be a balanced bipartite graph with classes $A$ and $B$. The graph $G$ contains a unique perfect matching $M$ if and only if there is an ordering $A=\{a_1,a_2,\dots, a_n\}$ and $B=\{b_1,b_2,\dots, b_n\}$ such that for all $i \in [n]$ we have $a_ib_i \in M$  and $a_ib_j \notin E(G)$ if $1 \leq j <i \leq n$.
\end{obs}
\begin{proof}
	If there are zero or at least two perfect matchings in $G$, such an ordering cannot exist.
	If there is a unique perfect matching, then one of the vertices must have degree one.
	We pick this vertex to be $a_n$ or $b_1$, depending which class it is from, and its neighbor by $b_n$ or $a_1$, respectively.
	After deleting this degree one vertex and its neighbor the remaining graph still has a unique perfect matching, so we can find an ordering of the other vertices by induction.
\end{proof}

We now transform a given bipartite graph $G$ with classes $A$ and $B$ into a weighted DAG such that every vertex is a source or a sink. 
First orient all edges in $G$ such that they are directed to $B$. Then assign weight $-1$ to every vertex of $A$ and weight $1$ to every vertex of $B$. 
Finally, add an isolated vertex, $v$, with weight $1$ to $G$. 
If $G$ is extendable to a bipartite graph with a unique perfect matching, then $v$ together with the order given by Observation \ref{parositas} gives a non-negative topological ordering of $G$. In particular, $v, a_1, b_1, a_2, b_2, \dots, a_n, b_n$ is a non-negative topological ordering of $G$. Furthermore, any non-negative topological ordering on $A$ and $B$ satisfies the requirements of Observation~\ref{parositas} (after adding the possibly missing $a_ib_i$ edges).

It was noted by Vialette~\footnote{Personal communication.} that the following problem can be reduced to Problem~\ref{graf} by adding $n-k$ isolated vertices to each class.

\begin{problem}\label{grafk}
	Let $G$ be a balanced bipartite graph with class sizes $n$ and let $k$ be a positive integer.
	Determine whether $G$ has an induced subgraph $H$ with $k$ vertices from each class of $G$ such that edges can be added to $H$ to create a bipartite graph with a unique perfect matching.
\end{problem}

To complete the proof of Theorem \ref{main-two}, we must show Problem \ref{grafk} is \NPcx.
This was done (for practically the same problem) by Dasgupta, Jiang, Kannan, Li and Sweedyk~\cite{D} using a reduction from the \NPc LARGEST BALANCED INDEPENDENT SET problem.\footnote{In \cite{GJ} the equivalent problem of finding the largest BALANCED COMPLETE BIPARTITE SUBGRAPH is shown to be \NPcx.}
We include a less technical proof of their reduction.

We call an independent set in a bipartite graph \emph{balanced} if each class of the bipartite graph contains exactly half of the vertices of the independent set.
We first state the LARGEST BALANCED INDEPENDENT SET problem.

\begin{problem}\label{indep}
	Let $G$ be a bipartite graph and let $k$ be a positive integer.
	Determine whether $G$ contains a balanced independent set on $2k$ vertices.
\end{problem}

Let $G$ be a bipartite graph with classes $A$ and $B$ and let $k$ be a positive integer. Construct a bipartite graph $G'$ as follows.
The vertex set of $G'$ consists of $k+1$ copies of each vertex $v$ in $G$, denoted by pairs $(v,1),(v,2),\dots, (v,k+1)$. We connect two vertices $(u,i)$ and $(v,j)$ in $G'$ by an edge if either of the following are satisfied:

(1) $u \in A$ and $v \in B$ and $i < j$.

(2) $uv \in E(G)$.

\begin{claim}\label{UPM}
	The graph $G'$ has a subgraph $H$ on $2k^2+2k$ vertices such that edges can be added to $H$ to create a bipartite graph with a unique perfect matching if and only if $G$ has a balanced independent set with $2k$ vertices.
\end{claim}

\begin{proof}
	If $G$ has a balanced independent set with $2k$ vertices, then call $H$ the induced subgraph of $G'$ spanned by the $k+1$ copies of this independent set. Clearly, $H$ has $2k^2+2k$ vertices. Furthermore, it is easy to see that adding the edges of a matching to each copy of the independent set results in a bipartite graph with a unique perfect matching.
	
	Now suppose that $G'$ has a subgraph $H$ on $2k^2+2k$ vertices such that edges can be added to $H$ to create a bipartite graph with a unique perfect matching. Let $A_H$ and $B_H$ be the two classes of $H$ defined by the partition of $G$. Now order the vertices of $A_H$ and $B_H$ by Observation~\ref{parositas} such that if $a<b$, then there is no edge between the $a$th vertex in $A_H$ and the $b$th vertex in $B_H$.
	
	Let $(w,i)$ be the first vertex in the ordering of $A_H$ such that among the vertices that appear earlier in the ordering there are $k-1$ different values in the first coordinate. Let $a$ be the index of $(w,i)$ in the ordering of $A_H$. 
	Let $m$ be the smallest value among the second coordinates of the vertices with index less than $a$ in $A_H$. Thus we have $a \leq (k-1)(k+1-m+1)$. Therefore $a \geq k^2+k-(k-1)(k+1-m+1)-1 = m(k-1) + 1$.
	
	Recall that if $b>a$, then there is no edge between the $a$th vertex in $A_H$ and the $b$th vertex in $B_H$. Furthermore, by (1), for $i<j$ there is an edge between each $(u,i) \in A_H$ and $(v,j) \in B_H$. Thus every vertex in $B_H$ with index $b>a$ has second coordinate at most $m$. There are at least $m(k-1) + 1$ vertices in $B_H$ with index $b>a$. Therefore, by the pigeonhole principle, there is a set of $k$ of these vertices with the same second coordinate. In particular, these $k$ vertices represent distinct vertices in the original graph $G$. By the definition of $(w,i)$ there is a set of $k$ vertices in $A_H$ with index at most $a$ and distinct first coordinates, i.e., distinct vertices in $G$. These two sets of vertices form an independent set of size $2k$ in $G$.
	This completes the proof of Claim~\ref{UPM}.
\end{proof}

Thus we have established that Problem~\ref{grafk} is \NPc and therefore we have proved Theorem \ref{main-two}.

We end this section with another problem posed by the fourth author\footnote{Positive topological ordering, take 3,	\url{http://cstheory.stackexchange.com/questions/1842}.} that is equivalent to Problem \ref{graf}, and thus also \NPcx. 

\begin{problem}\label{matrix}
	Suppose $M$ is an $n \times n$ matrix.
	Determine whether it is possible to reorder its rows and columns such that we get an upper-triangular matrix.
\end{problem}

The equivalence of Problems \ref{graf} and \ref{matrix} is as follows.
Define a bipartite graph with classes $A$ and $B$ from $M$ by letting $A$ be the rows of $M$ and $B$ be the columns of $M$, with an edge between two vertices if and only if the corresponding entry of $M$ is non-zero. 
It follows from Observation \ref{parositas} that we can reorder the rows and columns of $M$ to get an upper-triangular matrix if and only if we can extend $G$ to a bipartite graph with a unique perfect matching. 
The other direction is shown similarly.



\section*{Acknowledgments}
We would like to thank A.~Frank, T.~Kir\'aly and L.~V\'egh for bringing Problem~\ref{main-two-prob} to our attention and G.~Rote for useful discussions about the equivalence of Problem~\ref{grafk} and Problem~\ref{graf}.

\section*{Addendum}
Since this paper has gone to press it has come to the attention of the authors that Problem~\ref{matrix} has been also shown to be \NPc by Fertin, Rusu and Vialette~\cite{FRV}. They also observed (as we do in Section~\ref{NPC-section}) that the proof in \cite{D} is not complete.
Furthermore, they point out that Problem~\ref{matrix} had been asked earlier by Wilf~\cite{W}.


\end{document}